\documentclass[12pt]{amsart}

\usepackage{amsfonts, amssymb, amscd}

\def\dual                 {{\vee}}
\def\rk                 {{\rm rk}}
\def\ii                 {{\rm i}}
\def\ee                 {{\rm e}}

\def\BH		{{\rm Berglund-H\"ubsch }}
\def\BB		{{\rm Batyrev-Borisov }}
\def\F		{{\rm Fock}}

\def\ZZ                 {{\mathbb Z}}

\def\CC                 {{\mathbb C}}
\def\QQ                 {{\mathbb Q}}

\def\La    {{\mathcal L}}

\newtheorem{lemma}{Lemma}[subsection]
\newtheorem{theorem}[lemma]{Theorem}
\newtheorem{corollary}[lemma]{Corollary}
\newtheorem{proposition}[lemma]{Proposition}
\theoremstyle{definition}
\newtheorem{definition}[lemma]{Definition}

\newtheorem{remark}[lemma]{Remark}
\theoremstyle{remark}
\newtheorem*{proof*}{Proof}

\title{Berglund-H\"ubsch mirror symmetry via vertex algebras}

\author{Lev A. Borisov}

\address{Mathematics Department, Rutgers University, 110 Frelinghuysen Rd,
Piscataway, NJ 08854, email:borisov@math.rutgers.edu}

\begin{document}

\begin{abstract}We give a vertex algebra proof of the Berglund-H\"ubsch duality 
of nondegenerate invertible potentials.
We suggest a way to unify it with the Batyrev-Borisov duality of reflexive Gorenstein
cones.
\end{abstract}

\maketitle

\section{Introduction}\label{sec.intro}
Ever since its been first discovered in the early 1990's, mirror symmetry 
served as an inspiration to algebraic and symplectic geometers.  The 
original physical motivation behind it centers around the notion of 
$N=(2,2)$ superconformal field theory, which is a very rich and only
partially axiomatized structure. There are \emph{physical} methods
of assigning such theories to various combinatorial and/or geometric data.
In certain instances, the theories one obtains from two different 
sets of data are isomorphic via a very special \emph{mirror involution}, 
which leads to a deep connection between the two sets of data. The earliest 
example was the prediction of the (virtual) number of rational curves 
of given degree on a generic quintic threefold \cite{mirror}.

\medskip
Most classical treatments of mirror symmetry revolve around 
the calculation of the so-called $A$ and $B$ chiral
rings which are particular substructures of $N=(2,2)$ superconformal field 
theory. A mirror set of data is characterized by the property 
that the two rings are interchanged. The $A$ ring of
one set of data is supposed to be isomorphic to the $B$ ring
of the mirror set of data and vice versa. 
In many cases these two rings can be constructed mathematically from
the initial data, even while the whole theory can only be constructed physically. 
For example, one can start with a Calabi-Yau manifold $X$ with a 
complexified K\"ahler class $[w]$. Then the $A$ ring is 
the quantum cohomology of $X$ with the parameters
specialized to $[w]$. The $B$ ring is the cohomology of 
the exterior algebra of the tangent bundle of $X$.

\medskip
In the examples of interest, the $A$ and $B$ rings of the 
theory come with a double grading and with a natural
vector space isomorphism between them. If $\hat c>0$ 
is the central charge of the theory (equal to
the dimension of $X$ in the above example),
then this isomorphism  
sends a $(p,q)$-graded piece of the $A$ ring to the $(\hat c-p,q)$ graded piece
of the $B$ ring. Of course, this isomorphism is not compatible
with the product structures.

\medskip
Mirror symmetry construction of Batyrev \cite{Bat.dual}
and its generalization to Calabi-Yau complete intersections in toric varieties
known as \BB construction is one the best understood settings
of mirror symmetry. In contrast, the \BH construction of \cite{BeHu}
has unfortunately received too little attention, despite its 
appealing simplicity. Recently, papers of Krawitz \cite{Krawitz} 
and Chiodo-Ruan \cite{Chiodo-Ruan} partially corrected this
injustice by proving the conjectural mirror symmetry of \BH
at the level of double graded dimensions, and by connecting 
the Landau-Ginzburg version of the theory to the geometric
notion of orbifold cohomology.  Importantly, Krawitz was able to
define the dual potential and group of \BH construction in full 
generality, beyond the original examples of Berglind and H\"ubsch.
This paper hopes to further advance the understanding of 
Berglund-H\"ubsch-Krawitz duality
and to also sketch a path that will likely lead to a combined
setting that includes both \BB and \BH constructions as special 
cases. Specifically, we find that the vertex algebra approach to mirror 
symmetry, originally developed for \BB setting in \cite{borvert}
and \cite{chiralrings} can be applied with some modifications
to the \BH setting. It allows us to reprove the result of \cite{Krawitz}
and moreover to show the ring isomorphism between the $A$ ring
of  a \BH potential and the $B$ ring of the dual potential (with the 
appropriate choices of orbifoldizations).
In the process it leads to a natural combinatorial
setup for the unification of the two constructions.

\medskip
The vertex algebra approach to mirror symmetry aims
to construct a larger algebraic structure which contains
the $A$ and $B$ rings of the theory as two subspaces,
with the induced structure of supercommutative double
graded rings. The mirror set of data gives the same
larger structure, up to a natural involution that has the 
effect of switching $A$ and $B$ rings. 
Specifically, this larger algebraic structure is known as vertex (in some
treatments \emph{chiral}) algebra with $N=2$ structure, 
and the $A$ and $B$ rings are the chiral rings of this algebra,
in the sense of \cite{LVW}.
From the physics view point this vertex algebra is the state space of the half-twisted
theory. We recall the definition of these algebras in Section
\ref{sec.vertprelim}. The aforementioned
isomorphism of $A$ and $B$ rings of \emph{the same theory}
comes from a natural \emph{physical} construction called spectral flow.
The original calculations of \cite{borvert} have been heavily inspired by 
the chiral de Rham complex construction of Malikov, Schechtman and Vaintrob,
see \cite{MSV}. However, the setting of this paper is more combinatorial, since 
the underlying geometry is often not clear.

\medskip
The structure of the paper is as follows. 
In Section \ref{sec.BH-BB-comb} we recall the definitions of \BH potentials,
following \cite{Krawitz}, and then rephrase the construction
in the terms that are similar to those of \BB duality. Specifically, we encode
a pair of dual potentials and groups $(W,G)$ and $(W^\dual,G^\dual)$ by a pair of dual 
lattices $M$ and $N$ and collections of elements $\Delta$ and 
$\Delta^\dual$ in these lattices. The pairings of elements of $\Delta$ and $\Delta^\dual$
encode the matrix of degrees that occurs in the definition of \BH potentials.
The elements of $\Delta$ and the elements of $\Delta^\dual$
generate cones $K_M$ and $K_N$ respectively. While $K_M$ and $K_N$ 
are not quite dual to each other (as would be the case in \BB setting) they are close to 
being dual, because of the nondegeneracy of the potentials.

\medskip
In Section \ref{sec.V} we give a novel description of $A$ and $B$ rings
of \BH construction by what essentially amounts to a switch to 
logarithmic coordinates. We describe the $A$ and $B$ rings of $(W,G)$ as 
cohomology of the following complex, which is analogous to the one in \cite{BM}.

\medskip\noindent
{\bf Proposition \ref{A-iso}.}
The $A$ ring of \BH construction, with components twisted
by certain one-dimensional spaces, is naturally isomorphic 
to the cohomology of 
$$
\CC[(K_N^\dual\oplus K_N)_0]\otimes \Lambda^*(M_\CC)
$$
with respect to 
$$
d^A:= \sum_{m\in\Delta} [m]\otimes (\wedge m)
+\sum_{n\in \Delta^\dual} [n]\otimes ({\rm contr.} n).
$$

\medskip
Here $\CC[(K_N^\dual\oplus K_N)_0]$ is the quotient of the 
semigroup ring of $K_N^\dual\oplus K_N$ by the ideal generated by 
monomials $m\oplus n$ with $m\cdot n >0$. 
The $B$ ring is similarly described in Proposition \ref{B-iso} by replacing $\Lambda^*(M_\CC)$
$\Lambda^*(N_\CC)$ and switching contractions and exterior multiplications.
We also observe that the natural double grading on these spaces, which was first 
considered in \cite{BM} coincides with the double grading of $A$ and $B$ 
rings described in \cite{Krawitz}, following the more general work of Kaufmann
\cite{Kau3}.
In the language of Section \ref{sec.V}, the difference between the $A$ ring of $(W,G)$ and the $B$
ring of $(W^\dual,G^\dual)$ amounts to a switch from $\CC[(K_N^\dual\oplus K_N)_0]$ to 
$\CC[(K_M\oplus K_M^\dual)_0]$. While Sections \ref{sec.BH-BB-comb} and 
\ref{sec.V} are free of any vertex algebra techniques and are rather elementary,  we will
need the full machinery of vertex algebras to construct a natural isomorphism between these 
spaces.

\medskip
In Section \ref{sec.vertprelim} we give a brief review of vertex algebras and $N=2$
structures. We also recall lattice vertex algebras $\F_{M\oplus N}$ constructed
from pairs of dual lattices $M$ and $N$. 

\medskip
In Section \ref{sec.BB-BH-voa} we define the vertex algebras of \BH mirror symmetry 
and prove their first properties. Specifically, for a pair $(f,g)$ of generic coefficient functions
$f:\Delta\to\CC$ and $g:\Delta^\dual\to \CC$ we define
$V_{f,g}$ to be the cohomology of the lattice vertex algebra $\F_{M\oplus N}$ by 
the differential
$$
D_{f,g}:={\rm Res}_{z=0}(\sum_{m\in\Delta}f(m) m^{ferm}(z)
\ee^{\int m^{bos}(z)}
+\sum_{n\in \Delta^\dual} g(n)n^{ferm}(z)
\ee^{\int n^{bos}(z)}).
$$
We prove an analog of the Key Lemma of \cite{borvert} which states that $\F_{M\oplus N}$ 
can be replaced by $\F_{K_M\oplus N}$ or $\F_{M\oplus K_N}$. This is a nontrivial
consequence of the nondegeneracy of the potential, and it is absolutely crucial
for the subsequent discussion.

\medskip
In Section \ref{sec.main} we use the results of Section \ref{sec.BB-BH-voa} and 
\cite{chiralrings} to show that each of the 
chiral rings of $V_{f,g}$ is naturally isomorphic to the cohomology of a pair of complexes.
When $(f,g)={(\bf 1,\bf 1})$ these complexes are precisely the ones
considered in Section \ref{sec.V}, which shows the isomorphism between the $A$ and
$B$ rings of the dual \BH potentials.
The main result of the paper is Theorem \ref{main}.

\medskip\noindent
{\bf Theorem \ref{main}.}
For generic choices of $f$ and $g$
the $N=2$ vertex algebra $D_{f,g}$ is 
of $\sigma$-model type. 
In particular, this algebra is of 
$\sigma$-model type for $(f,g)=({\bf 1},{\bf 1})$.
The $B$ ring of $V_{{\bf 1},{\bf 1}}$ can be identified with 
the $B$ ring of $(W,G)$ and with the $A$ ring of
$(W^\dual,G^\dual)$. The $A$ ring of $V_{{\bf 1},{\bf 1}}$ can
be identified with the $A$ ring of $(W,G)$ and with the $B$ ring of
$(W^\dual,G^\dual)$.

\medskip
One advantage of our method, as compared to that of \cite{Krawitz} is that it does not  make
explicit use of the classification of invertible potentials of Kreuzer and Skarke \cite{KS}. 
This does not make our argument easier,  given the extensive use of the vertex algebra 
techniques. However, it does make our approach more natural and it 
allows us to attempt to unify \BH and \BB constructions by 
assuming sufficient conditions that ensure that the argument of Sections 
\ref{sec.BB-BH-voa} and \ref{sec.main} is stil applicable. This unification is explored
in Section \ref{sec.unification}. The condition is a generalization of nondegeneracy
of the potential to the case when the number of elements in $\Delta$ and $\Delta^\dual$
is larger than the rank of 
the lattice.
 We give two equivalent formulations in Definition \ref{unified}
and Proposition \ref{prop.unified}. Finally, in Section \ref{sec.open} we state several natural open problems.

\medskip
The reader should be forewarned regarding our 
notation choices. We deliberately use the notations
for \BH setting that are analogous to the \BB setting.
Eventually, the two pictures are joined in Section \ref{sec.unification}, but until that
point we try to make it clear from the context which setting
we are considering.

\medskip
\emph{Acknoledgements.} I would like to thank Max Kreuzer 
for introducing me to the \BH construction and for useful comments
on the first version of the paper. The preprints
\cite{Chiodo-Ruan} and \cite{Krawitz} were also a major influence.
I would like to thank Ralph Kaufmann for useful comments on the first 
version of the paper. This work was partially supported by the NSF Grant 1003445.

\section{Berglund-H\"ubsch and Batyrev-Borisov
 constructions}\label{sec.BH-BB-comb}

The goal of this section is to restate the \BH construction
of mirror potentials in the terms that are similar to Batyrev's 
reflexive polyhedra construction and, more generally, 
to \BB dual reflexive cones construction.

\subsection{Berglund-H\"ubsch potentials}
The mirror symmetry construction of Berglund and H\"ubsch
begins with a nondegenerate polynomial potential
$$W(x_1,\ldots,x_d)=\sum_{i=1}^d c_i\prod_{j=1}^d x_j^{a_{ij}}$$
with invertible integer matrix $(a_{ij})$,
where the coefficients $c_i$ can be assumed to equal $1$ 
without loss of generality. The variables $x_j$ are assumed
to be given positive rational degrees  $q_j$ which make
$W$ homogeneous of degree $1$, i.e. we have
$$
\sum_{j=1}^d a_{ij}q_j = 1 
$$
for all $i$. Nondegeneracy of the potential means that the 
hypersurface $W=0$ in $\CC^d$ is smooth away from
the origin. It is a very restrictive condition, and complete 
classification of nondegenerous potentials is given in \cite{KS}.

\medskip
Consider the group ${\rm Aut}(W)$ of diagonal automorphisms 
\begin{equation}\label{gamma}
\gamma:x_j\mapsto \gamma_j x_j
\end{equation}
which preserve the potential $W$. This is a finite abelian group,
which contains two natural subgroups. First,
${\rm SL}_d\cap {\rm Aut}(W)\subseteq {\rm Aut}(W)$ is 
characterized by $\prod_j \gamma_j=1$. Second, there is a subgroup of ${\rm Aut}(W)$ 
generated by the automorphism
$$
x_j\mapsto \exp(2\pi\ii q_j) x_j
$$
which is called the exponential grading operator in \cite{Krawitz}.
We will be mostly interested in the case when the second 
of these groups is contained in the first one, which means
that the exponential graded operator acts with determinant one.
This translates into
\begin{equation}\label{CYBH}
\sum_{j=1}^d q_j =k\in \ZZ_{>0}.
\end{equation}

\begin{remark}
Condition \eqref{CYBH} is a  generalization of the Calabi-Yau
condition $\sum_{j=1}^d q_j =1$, see \cite{Chiodo-Ruan}. 
The case $\sum_{j=1}^d q_j=k$
will turn out to be analogous to Batyrev-Borisov construction
for Calabi-Yau complete intersections of $k$ hypersurfaces. We will thus 
continue to refer to the above condition as the Calabi-Yau condition.
\end{remark}

As part of the combinatorial data of \BH construction one
fixes a subgroup $G\subseteq {\rm Aut}(W)$. 

\subsection{Dual potential and dual group}
It follows from the classification of \cite{KS} that if
$$
W=\sum_{i=1}^d\prod_{j=1}^{d}x_j^{a_{ij}}
$$
is a nondegenerate potential, then 
$$
W^\dual=\sum_{i=1}^d\prod_{j=1}^{d}y_j^{a_{ji}}
$$
is also nondegenerate.

\medskip
Let $G$ be a subgroup of ${\rm Aut}(W)$. Then Krawitz
\cite{Krawitz} defines the dual group $G^\dual\subseteq {\rm Aut}
(W^\dual)$. We describe his construction while trying to keep
our notations as compatible as possible.
Define $A_W=(a_{ij})$. Then consider
the inverse matrix $A_W^{-1}$ and define $\rho_i$ and 
$\bar\rho_i$ to be its $i$-th column and $i$-th row respectively.
Abusing the notation slightly, we will also denote by 
$\rho_i$ the corresponding element of ${\rm Aut}(W)$ 
obtained by multiplying $x_j$ by $\exp(2\pi\ii (\rho_i)_j)$.
Similarly, $\bar \rho_j$ will also be used to denote an element
of ${\rm Aut}(W^\dual)$.

\begin{definition}\cite{Krawitz}\label{KrawitzDual}
The dual group $G^\dual$ is defined by
$$
G^\dual:= \left\{
\prod_{i=1}^d \bar \rho_i^{r_i}~\Big\vert~
[r_1,\cdots,r_d]A_W^{-1}\left[\begin{array}{c}a_1\\
\vdots \\ a_d 
\end{array}\right]
\in \ZZ~{\rm for~all~}\prod_{i=1}^d \rho_i^{a_i}\in G
\right\}.
$$
\end{definition}

\subsection{Combinatorial reformulation of the Berglund-H\"ubsch potential and group}
\label{ssec.combref}

We will now reformulate the data of Berglund-H\"ubsch potential
$W$ and the group $G$ in combinatorial terms that resemble
the setting of Batyrev-Borisov mirror symmetry construction.

\medskip
Consider the free abelian groups $M_0$ and $N_0$ 
with bases $(u_i),i=1,\ldots,d$ and $(v_j),j=1,\ldots,d$.
Define an integer, nondegenerate (but not unimodular) pairing on these lattices by 
$$u_i\cdot v_j := a_{ij}$$
where $a_{ij}$ are the exponents of the potential $W$, as
in the previous subsection.

\begin{proposition}\label{BHtoric}
The choice of the group $G$ and its dual $G^\dual$ is
naturally equivalent to the choice of  suplattices $M\supseteq M_0$
and $N\supseteq N_0$ such that the induced pairing 
on $M$ and $N$ is integer and unimodular, i.e. $M$ and $N$
are dual lattices.
\end{proposition}

\begin{proof}
First observe that ${\rm Aut}(W)$ is naturally isomorphic
to $M_0^\dual/N_0$. The isomorphism sends 
$\gamma=(\gamma_i)$ in \eqref{gamma} to 
$N_0+\sum_{j=1}^d\frac 1{2\pi \ii} (\log \gamma_j ) v_j$.
Consequently, a choice of $G$ means a choice of 
lattice $N$ with $N_0\subseteq N\subseteq M_0^\dual$.
We have a similar identification of ${\rm Aut}(W^\dual)$
with $N_0^\dual/M_0$. It remains to show that the 
definition of $G^\dual$ in the previous subsection
is equivalent to considering $M=N^\dual$.
Indeed, $\rho_i$ correspond to the elements of the 
basis of $M_0^\dual$ which is dual to the basis $(e_i)$ of $M$
and $\bar\rho_j$ correspond to the basis of $N_0^\dual$
which is dual to the basis  $(v_j)$ of $N$. The pairing matrix
$(\bar\rho_i\cdot \rho_j)$ is equal by $A_W^{-1}$ and the 
rest is clear.
\end{proof}

\begin{remark}
The Berglund-H\"ubsch duality now simply means interchanging
the two lattices $M$ and $N$ together with the sets of
chosen elements $(u_i)$, $(v_j)$ in them. It is now obvious
that $(G^\dual)^\dual=G$.
\end{remark}

We can now give a simple combinatorial interpretation of 
the two subgroups of ${\rm Aut}(W)$ and the Calabi-Yau
condition.
\begin{definition}
We define elements $\deg\in N_0^\dual$ and 
$\deg^\dual\in M_0^\dual$ by 
$
u_i\cdot \deg^\dual = 1, ~\deg\cdot v_i = 1, ~{\rm for ~ all~}i
$.
\end{definition}

In the following proposition we implicitly assume the identification
of the data of $G$ and $G^\dual$ with the dual suplattices 
$M$ and $N$ in Proposition \ref{BHtoric}.
\begin{proposition}
The group $G$ lies in ${\rm SL}_d\cap {\rm Aut}(W)$ 
if and only if $\deg\in M$. The group $G$ contains 
the exponential grading operator if and only if $\deg^\dual\in N$.
The Calabi-Yau condition is equivalent to $\deg\cdot \deg^\dual
\in Z_{> 0}$.
\end{proposition}

\begin{proof}
The condition that $\deg\in M$ is equivalent to the statement
that $\deg$ has integer pairing with every 
$\sum_{j=1}^d\frac 1{2\pi \ii} (\log \gamma_j ) v_j$
for $\gamma=(\gamma_i)\in G$. This pairing is 
$\sum_{j=1}^d \frac 1{2\pi \ii} (\log \gamma_j )$, and its
integrality is equivalent to $\prod_j \gamma_j=1$.
The element $\deg^\dual$ equals
$\sum_j q_j v_j$, which in turns corresponds to the exponential 
grading operator.
Finally, $\deg\cdot \deg^\dual = \deg \cdot \sum_j q_j v_j
=\sum_j q_j$.
\end{proof}

\begin{corollary}
The group $G$ lies in $SL_d$ if and only if $G^\dual$ 
contains the exponential grading operator. The Calabi-Yau condition means
that there exists $G$ such that the corresponding 
lattices $M$ and $N$ satisfy both
$
\deg\in M
$
and $\deg^\dual\in  N$.
\end{corollary}

\begin{proof} Clear. \end{proof}

\begin{remark}
In what follows we will be mostly interested in the case when 
the group $G$ contains the exponential grading
operator and lies in ${\rm SL}_d$. In particular,
the potential $W$ must satisfy the Calabi-Yau condition.
We will refer to this situation as \BH construction of 
Calabi-Yau type of index $k$ where $k=\sum_jq_j$.
\end{remark}

We introduce the following notations to highlight the 
analogy between the combinatorial data of \BH
and \BB constructions.
\begin{definition}
Define the cones $K_M$ in $M$ and $K_N$ in $N$
by 
$$K_M:=\sum_{i} \QQ_{\geq 0}u_i,
~
K_N:=\sum_{j} \QQ_{\geq 0}v_j.
$$
\end{definition}

\begin{proposition}\label{defcones}
Consider the dual cones $K_M^\dual$ in $N$ and
$K_N^\dual$ in $M$. Then we have 
$K_M\subseteq K_N^\dual$ and 
$K_N\subseteq K_M^\dual$. 
\end{proposition}

\begin{proof} The statement follows from $a_{ij}\geq 0$. \end{proof}

\begin{remark}
In the case of Fermat type potential (see \cite{KS}),
we have equalities in Proposition \ref{defcones}.
In general, the inclusions are strict. However, the 
nondegeneracy of the polynomial potential $W$ may
be viewed as the statement that the complements
are in some sense small.
\end{remark}

\begin{definition} Denote by $\Delta$ the set
$(u_i)$ and by $\Delta^\dual$ the set $(v_j)$.
\end{definition}

\begin{remark}
We may associate arbitrary nonzero elements $c_i$ to
the elements $u_i\in \Delta$, as in the original definition
of $W$. Similarly, we can associate arbitrary elements
to monomials in $W^\dual$.  These can all be rescaled 
to $1$.
\end{remark}

\subsection{Comparison with Batyrev-Borisov construction}
Recall that the discrete
combinatorial data of Batyrev-Borisov construction
consist of the pair of dual lattices $M$ and $N$ 
and dual reflexive Gorenstein cones $K$ in $M$ and $K^\dual$ 
in $N$.

\begin{definition} Let $M$ and $N$ be dual lattices.
The dual rational polyhedral cones $K$ in $M$ and $K^\dual$ in $N$
are called dual Gorenstein if there exist elements $\deg\in M$ 
and $\deg^\dual\in N$ so that the lattice generators $u$ of the rays of 
$K$ satisfy $u\cdot \deg^\dual=1$ and lattice generators $v$ 
of rays of $K^\dual$ satisfy $\deg\cdot v=1$.
\end{definition}

\begin{definition}\cite{chiralrings}
Define $\Delta$ to be the set of lattice elements $m\in K$ 
which satisfy $m\cdot \deg^\dual=1$ and $\Delta^\dual$
to be the set of lattice elements $n\in K^\dual$ such that
$\deg\cdot n=1$.
\end{definition}

\begin{definition}
The index of the pair of dual reflexive Gorenstein cones
is defined as $\deg\cdot\deg^\dual$. 
\end{definition}

\begin{remark}
The case of index one corresponds to the original
Batyrev duality, see \cite{Bat.dual}. The higher index $k$
case may or may not have a geometric underpinning.
In good cases, one side of the duality may
be related to a Calabi-Yau complete 
intersection of $k$ hypersurfaces. In the best case scenario,
both sides are related to a complete intersection. This best case
scenario corresponds to the nef-partition case, considered in 
\cite{Borisov.preprint}. The issue is whether $\deg$ 
(resp. $\deg^\dual$) lie in the semigroup generated by $\Delta$
(resp. $\Delta^\dual$).
The reader is referred for more 
details to the paper of Batyrev and Nill \cite{Bat.Nill}.
\end{remark}

Another important ingredient of \BB construction is 
a pair of coefficient functions, which are generic 
functions $f:\Delta\to \CC$ and $g:\Delta^\dual\to \CC$.
In the good cases they correspond to the coefficients 
of the defining equations of a pair of 
elements of mirror families. 
\BB duality corresponds to simply switching
all the data and their duals.

\medskip
It should now be clear to the reader that the above 
formulation of \BH construction of  Calabi-Yau type of
index $k$  \emph{almost}
fits into the framework of \BB construction of index $k$. Specifically,
we have dual lattices $M$ and $N$, 
elements $\deg\in M$ and $\deg^\dual\in N$,
finite sets $\Delta$ and $\Delta^\dual$, with
$\Delta\cdot \deg^\dual=1$, $\deg\cdot\Delta^\dual=1$,
$\deg\cdot\deg^\dual=k$,
as well as generic coefficient functions $f$, $g$. 
The only difference is that in the \BH case the cones 
$K_M=\sum_{m\in\Delta}\QQ_{\geq 0} m $
and $K_N = \sum_{n\in\Delta^\dual}\QQ_{\geq 0}n$
are not quite dual to each other but only satisfy
$K_M\cdot K_N\geq 0$.
In what follows, we will pursue this parallel to its
fullest by applying to the \BH setting 
the ideas and results of the vertex 
algebra approach, developed in the context of \BB mirror construction. 
Afterwards we will attempt to unify the two constructions
in Section \ref{sec.unification}.

\medskip
We will first pursue the analogy of \BH and \BB constructions
at the level of $A$ and $B$ rings, avoiding the vertex algebra machinery.
However, we will eventually invoke the full
machinery of the vertex algebra approach to prove the 
\BH duality.

\section{The double-graded cohomology spaces}\label{sec.V}

In this section, we reinterpret the $A$ and $B$ rings of \BH  
construction as the cohomology
of some double-graded complexes. These complexes are 
motivated by the calculation of $A$ and $B$ rings 
for \BB construction in \cite{BM}. 

\subsection{Calculation of \cite{BM} of the $A$ and $B$ rings of  Batyrev-Borisov construction.}
Let $M$ and $N$ be dual lattices, $K$ in $M$ and $K^\dual$ in $N$
a  pair of dual Gorenstein cones, $\deg\in M$, $\deg^\dual\in N$
the corresponding degree elements, and $\Delta\subset M$
and $\Delta^\dual\subset N$ the degree one elements of 
the appropriate cones. Let $f$ and $g$ be generic coefficient
functions. 

\medskip
Consider the vector space $\CC[(K\oplus K^\dual)_0]$
which has a basis indexed by the pairs of lattice points 
$(m,n)$ in $(K,K^\dual)$
with $m\cdot n = 0$. We will denote the corresponding 
monomials by $[m+n]$. This vector space has a natural 
structure of the module over the semigroup ring $\CC[K\oplus
K^\dual]$, since it is the quotient of $\CC[K\oplus K^\dual]$
by the ideal generated by $[m+n]$ with $m\cdot n>0$.
The following description of the $A$ and $B$ rings of \BB mirror
symmetry construction has been suggested in \cite{BM}
and \cite{chiralrings}.
\begin{definition}\label{BMchiral}
Consider the space 
$$\CC[(K\oplus K^\dual)_0]\otimes_\CC\Lambda^*(M_\CC)$$
with endomorphism $d^A_{f,g}$ given by 
$$
d^A_{f,g}:= \sum_{m\in\Delta} f(m)[m]\otimes (\wedge m)
+\sum_{n\in\Delta^\dual}g(n)[n]\otimes({\rm contr.} n)
$$
where ${\rm contr.}n$ means taking a contraction in
the exterior algebra of $M_\CC$ by element $n$.
The cohomology of $\CC[(K\oplus K^\dual)_0]\otimes_\CC\Lambda^*(M_\CC)$
with respect to $d^A_{f,g}$ is called the $A$  ring
associated to the data $(M,N,K,K^\dual,f,g)$.
Similarly, the cohomology of 
$$\CC[(K\oplus K^\dual)_0]\otimes_\CC\Lambda^*(N_\CC)$$
under 
$$
d^B_{f,g}:= \sum_{m\in\Delta} f(m)[m]\otimes ({\rm contr.} m)
+\sum_{n\in\Delta^\dual}g(n)[n]\otimes(\wedge n)
$$
is called the $B$  ring associated to these data.
\end{definition}

\begin{definition}\label{grading-A}
The $A$ ring possesses a natural
double grading as follows. There is the \emph{conformal}
grading induced from the grading on 
$\CC[(K\oplus K^\dual)_0]\otimes_\CC\Lambda^*(M_\CC)$
which assigns to $[m\oplus n] \otimes P$ the 
degree $-m\cdot \deg^\dual + \deg\cdot n + \deg (P)  -\deg\cdot 
\deg^\dual$. This grading is preserved by $d^A_{f,g}$. 
There is an additional \emph{cohomological}
grading induced by $m\cdot \deg^\dual + \deg\cdot n$ which
is increased by one by $d_{f,g}^A$.
\end{definition}

\begin{definition}\label{grading-B}
The $B$ ring possesses a natural
double grading as follows. The conformal
grading assigns to $[m\oplus n] \otimes P$ in 
$\CC[(K\oplus K^\dual)_0]\otimes_\CC\Lambda^*(N_\CC)$
the degree $m\cdot \deg^\dual - \deg\cdot n + \deg (P)  -\deg\cdot 
\deg^\dual$. This grading is preserved by $d^B_{f,g}$. 
The cohomological grading assigns to this element degree
$m\cdot \deg^\dual + \deg\cdot n-\deg\cdot \deg^\dual$. 
It is increased by one by $d_{f,g}^B$.
\end{definition}

\begin{remark}
The isomorphism between the $\Lambda^*(M_\CC)$
and $\Lambda^*(N_\CC)$ induces an isomorphism of 
the above complexes and thus an isomorphism of
$A$ and $B$ rings. This isomorphism clearly 
preserves the cohomological 
grading and replaces the conformal grading of $p$
by the conformal grading of $(\rk M-2\deg\cdot\deg^\dual)-p$.
The number $(\rk M-2\deg\cdot\deg^\dual)$ is the central
charge of the theory (and the dimension of the Calabi-Yau
variety if one can be associated to this set of data).
\end{remark}

The following key proposition has been proved in \cite{BM}.
\begin{proposition}
The dimensions of the double-graded components of 
the $A$-ring and the $B$-ring of the theory coincide with the 
stringy Hodge numbers of the Calabi-Yau variety,
if one can be associated to this set of data. In the particular
case of $\deg\cdot \deg^\dual = 1$, this Calabi-Yau variety
is ${\rm Proj}(\CC[K]/\langle\sum_{m\in\Delta}f(m)[m]\rangle)$.
\end{proposition}

It is not at all clear from the above description
how to construct a product on
the $A$ and $B$ rings above. In fact,
the only known general definition invokes the machinery
of vertex algebras in \cite{chiralrings}. 
However, there are some natural subrings
of $A$ and $B$ on which the product can be constructed
directly. 

\begin{remark}\label{diag}
Consider the subspace of 
$$\CC[(K\oplus K^\dual)_0]\otimes_\CC\Lambda^*(M_\CC)$$
spanned by the elements of the form 
$[{\bf 0}+n_1]\otimes 1$ where $n_1$ lies in the interior of $K^\dual$.
This subspace is annihilated by $d_{f,g}^A$, because 
$m\cdot n_1>0$ for all $m\in \Delta$. The image of this space in
the $A$ ring is isomorphic, as a $\CC$-algebra, to the quotient of 
$\CC[K^\dual]/\langle \sum_{n\in \Delta^\dual} g(n)
(m\cdot n )[n]\rangle $ by the annihilator of the element 
$[\deg^\dual]$.  Similarly, there is a natural subring
of the $B$ ring which comes from the linear combinations 
of the elements of the 
form $[m_1+{\bf 0}]\otimes 1$ with $m_1$ in the interior of $K$.
\end{remark}

\subsection{Preliminary lemma on Jacobian quotients}
In order to give a description of the $A$ and $B$ rings in the \BH
setting, we will first need a simple lemma which applies more
generally to any polynomial potential, not necessarily of \BH type.

\medskip
Let $q_i\in\QQ_{>0}$ denote
the weights of the variables $x_i$ in a polynomial ring of $d$
variables. Let $F(x_1,\ldots,x_d)$ be a total degree $1$ 
polynomial such that $F=0$ is smooth away from the origin.
Equivalently, the partial derivatives $\frac{\partial F}{\partial x_i}$
form a regular sequence. We are interested in the 
quotient 
$$\CC[x_1,\ldots,x_d]/\langle \frac{\partial F}{\partial x_1},
\ldots,\frac{\partial F}{\partial x_d}\rangle$$
which we will further shift by the product of the variables.
This additional shift is important for finite group actions that
need to be considered in the \BH construction.

\medskip
We denote by $\CC[{\bf x},{\bf y}]_0$ the quotient 
of the polynomial ring in $n$ variables by the monomials
$x_iy_i$. It has a basis of monomials whose $x$-support
is disjoint from $y$-support. Consider the space 
$$\CC[{\bf x},{\bf y}]_0\otimes \Lambda^*(\CC e_1+\ldots+ \CC e_d)$$
and the operator
$$
d^B_{f,{\bf 1}}:=\sum_{i=1}^d x_i\frac{\partial F}{\partial x_i}\otimes
{\rm contr.}e_i^\dual + \sum_{i=1}^d y_i \otimes\wedge e_i
$$
on it, where $e_i^\dual$ is the $i$-th element 
of the dual basis. 
\begin{lemma}\label{JacToLogJac}
Operator
$d^B_{f,{\bf 1}}$ is a differential. Its cohomology is naturally
isomorphic to 
$$
(\prod_{i=1}^d x_i)\CC[x_1,\ldots,x_d]/\langle 
\frac{\partial F}{\partial x_1},\ldots,\frac{\partial F}{\partial x_d}\rangle.
$$
\end{lemma} 

\begin{proof}
It is easy to see that the terms in $d^B_{f,{\bf 1}}$ anticommute
with each other. The only interesting case is the anticommutator
of the $i$-th term from the first sum and the $i$-th term from the 
second sum, which equals $x_iy_i\frac{\partial F}{\partial x_i}\otimes 1$, 
which is zero because we work in $\CC[{\bf x},{\bf y}]_0$.
We can also observe that $d^B_{f,{\bf 1}}$ increases by one
the total degree in ${\bf x}$ and ${\bf y}$ where $x_i$ has degree $q_i$
and $y_i$ has degree $1$.

\medskip
Consider the subcomplex $\mathcal C$ of 
$\CC[{\bf x},{\bf y}]_0\otimes \Lambda^*(\CC e_1+\ldots+ \CC e_d)$
which is spanned by elements of the form
\begin{equation}\label{P}
P({\bf x})\prod_{i\in I} x_i \otimes (\wedge_{i\not\in I}e_i). 
\end{equation}
where $I$ ranges over all subsets of $\{1,\ldots,d\}$ and
$P$ is any polynomial in $\CC[x_1,\ldots,x_d]$. The wedge terms 
of  $d^B_{f,{\bf 1}}$ act trivially on it, in view of $x_iy_i=0$. 
The contraction terms act exactly like the Koszul complex
for $\frac{\partial F}{\partial x_i}$ on $\CC[{\bf x}]$ on the $P$ 
parts of \eqref{P}. As a result, the cohomology of $\mathcal C$
is precisely
$$
(\prod_{i=1}^d x_i)\CC[x_1,\ldots,x_d]/\langle 
\frac{\partial F}{\partial x_1},\ldots,\frac{\partial F}{\partial x_d}\rangle.
$$
It remains to show that 
\begin{equation}\label{q}
\Big(\CC[{\bf x},{\bf y}]_0\otimes \Lambda^*(\CC e_1+\ldots+ \CC e_d)\Big)/{\mathcal C}
\end{equation}
has zero cohomology. Consider the double grading given by the degree in ${\bf x}$ 
and the degree in ${\bf y}$ where each $y_i$ is given degree $1$.
Note that the first and second sums $d_1$ and $d_2$ in 
$d_{f,{\bf 1}}^B$ change this bidegree by $(1,0)$ and $(0,1)$
respectively.
The stupid filtrations converge, because the dimensions of the 
graded pieces are finite. Thus it suffices to check that cohomology
of \eqref{q} with respect to $d_2$ vanishes.

\medskip
The cohomology of \eqref{q} with respect to $d_2$ 
can be computed separately 
for each multidegree in ${\bf x}$.  Consider $\prod_{i\in I}x_i^{r_i}$
with positive $r_i$. The corresponding part of 
\eqref{q} is 
$$\Big(\CC[y_i, i\not\in I]\otimes  \Lambda^*(\oplus_{i=1}^d\CC e_i)
\Big)/\Big(\Lambda_{j\not\in I}e_j\wedge \Lambda^*(\oplus_{i\in I}\CC e_i)\Big).
$$ 
This complex is quasi-isomorphic to the  tensor product of 
$ \Lambda^*(\oplus_{i\in I}\CC e_i)$
with trivial differential and the augmented Koszul complex
for $\CC[y_i,i\not\in I]$ and the sequence $(y_i,i\not\in I)$.
The latter is acyclic, which finishes the proof.
\end{proof}

We will now consider abelian group actions in the context of Lemma \ref{JacToLogJac}.
Let $G$ be a finite abelian group which 
acts on ${\bf x}$ diagonally and fixes the potential $F$.
Consider two dual lattices
$\ZZ^d$ with the standard cones $(\ZZ_{\geq 0})^d$
such that the corresponding semigroup algebras 
are identified with $\CC[x_1,\ldots, x_d]$ and $\CC[y_1,\ldots,y_d]$
respectively. Let $M$ be the sublattice that consists of degrees of
monomials in ${\bf x}$ that are fixed by $G$. Denote by $C$
the cone in $M$ which is the intersection of the standard cone
with it.  Denote by $\CC[(C\oplus (\ZZ_{\geq 0})^d)_0]$ 
the quotient of $\CC[C\oplus (\ZZ_{\geq 0})^d]$ by the 
monomials with positive pairing.
\begin{proposition}\label{fixed}
The $G$-fixed part of the (shifted) Milnor ring of $F$ is given
by the cohomology of 
$$
\CC[(C\oplus (\ZZ_{\geq 0})^d)_0]\otimes \Lambda^*(\oplus_{i=1}^d \CC e_i)
$$
with respect to
$$
d^B_{f,{\bf 1}}:=\sum_{i=1}^d x_i\frac{\partial F}{\partial x_i}\otimes
{\rm contr.}e_i^\dual + \sum_{i=1}^d y_i \otimes\wedge e_i.
$$
\end{proposition}

\begin{proof}
We first remark that the logarithmic partial derivatives 
$x_i\frac{\partial F}{\partial x_i}$ make sense as elements 
of $\CC[C]$. The action of $G$
on the shifted Milnor ring is induced from the diagonal action
on $x_i$. We consider the trivial action of $G$
on $y_i$ and $e_i$ to define the action on the complex of 
Lemma \ref{JacToLogJac}, which induces the same action on
the cohomology. It remains to observe that 
$\CC[(C\oplus (\ZZ_{\geq 0})^d)_0]$
is the $G$-invariant part of 
$\CC[{\bf x},{\bf y}]_0$.
\end{proof}

\subsection{Calculation of the $A$ and $B$ rings of Berglund-H\"ubsch construction}
We use the notations $W$, $G$, $M$, $N$, $K_M$, $K_N$,
$\deg$, $\deg^\dual$, $\Delta$, $\Delta^\dual$ from
Subsection  \ref{ssec.combref}. In this section we are only
concerned with the double-graded vector spaces, so the
ring structure is ignored.

\medskip
The following definition of the $B$  ring of \BH construction
follows from the definitions of \cite{Krawitz}.
\begin{definition}
The $B$  ring is given by 
$$
\Big(\oplus_{g\in G} \La_g \Big)^G
$$ 
where $\La_g^G$ is the $G$-invariant part of the Milnor ring
in the variables 
$x_i$ that are fixed by $g$, with respect to the restriction
of the potential, shifted by the product of variables fixed under $g$.
\end{definition}

We define $\CC[(K_N^\dual\oplus K_N)_0]$ to be 
the space with the basis indexed by $[m+n]$
where $m$ and $n$ are lattice points in $K_N^\dual$ and 
$K_N$ respectively. It is naturally a module over 
 $\CC[K_N^\dual\oplus K_N]$.
\begin{proposition}\label{B-iso}
There is a natural isomorphism between the $B$  ring
of \BH construction, with components twisted by certain
one-dimensional spaces, and the cohomology of 
$$
\CC[(K_N^\dual\oplus K_N)_0]\otimes \Lambda^*(N_\CC)
$$
with respect to 
$$
d^B:= \sum_{m\in\Delta} [m]\otimes ({\rm contr.}m)
+\sum_{n\in \Delta^\dual} [n]\otimes (\wedge n)
$$
$$
=\sum_{m\in\Delta} [m]\otimes ({\rm contr.}m)
+\sum_{j=1}^d[v_j]\otimes (\wedge v_j)
$$
where we use the combinatorial reformulation of 
Section \ref{sec.BH-BB-comb}.
\end{proposition}

\begin{proof}
Denote by $N_0$ the lattice generated by $v_j,j=1,\ldots,n$,
see Proposition \ref{BHtoric}. 
The complex under consideration splits into a direct sum of 
complexes that correspond to elements $N/N_0$, i.e. to elements
of $G$. 
 
\medskip
Consider one element  $g\in G$ and the corresponding 
rational linear combination $n_g=\sum_{j=1}^d h_j v_j$
with $0\leq h_j<1$. The corresponding action on $x_j$
is $x_j\mapsto \exp(2\pi\ii h_j)x_j$. 
The corresponding part of the complex will have $n$
of the form $n_g+\sum_{j=1}^d \ZZ_{\geq 0}v_j$.
The property $m\cdot n=0$ implies that only monomials
in the variables $x_j$ that are fixed under $g$ can appear
in the $g$-part of 
$\CC[(K_N^\dual\oplus K_N)_0]\otimes \Lambda^*(N_\CC)$.

\medskip
The addition to $n$ of $v_j$ for which $h_j>0$ thus does not 
affect the $m$ that can occur. The $g$-part of 
$\CC[(K_N^\dual\oplus K_N)_0]\otimes \Lambda^*(N_\CC)$
is then seen to be the tensor product of two complexes.
The first complex is given by 
$$
\CC[\sum_{j,h_j>0} \ZZ_{\geq 0} v_j]\otimes \Lambda^*
(\oplus_{j, h_j>0}\CC e_j)
$$
with the differential $\sum_{j,h_j>0} [v_j]\otimes (\wedge v_j)$.
The second complex is precisely the complex of Proposition 
\ref{fixed} for the group $G$ and the variables $x_j$ that are fixed
by $g$, with $[v_j]$ serving as $y_j$. These variables 
$x_j$ correspond to the elements of the dual basis to $(v_j)$.

\medskip
The cohomology of the first complex is one-dimensional
and is represented by $\CC\Lambda_{j,h_j>0} v_j$.
The cohomology of the second complex is given by 
Proposition \ref{fixed} as the $G$-equivariant part of $\La_g$.
This gives the natural isomorphism claimed in the statement of the proposition.
\end{proof}

\begin{remark}
The arguments of Lemma \ref{JacToLogJac} and Proposition
\ref{B-iso} imply that the $g$-part of the $B$ ring in our combinatorial
description comes from elements of the form
\begin{equation}\label{comesfrom}
[m\oplus n_g]\otimes\Lambda_{j, h_j>0} v_j
\end{equation}
where $n_g$ is the element that corresponds to $g$ and 
$m$ ranges over lattice elements in the interior of the face of
$K_N^\dual$ with $m\cdot n_g =0$.
\end{remark}

We will now show that the above isomorphism preserves
the structures of the double graded supervector spaces. We will first
describe the double grading and parity in our combinatorial 
formulation. We follow Definition \ref{grading-B}
and define the conformal and cohomological grading of 
$(m\oplus n)\otimes P$ in $\CC[(K_N^\dual\oplus K_N)_0]\otimes \Lambda^*(N_\CC)$ by
$m\cdot \deg^\dual - \deg\cdot n + \deg (P)  -\deg\cdot 
\deg^\dual$ and
$m\cdot \deg^\dual + \deg\cdot n-\deg\cdot\deg^\dual$ respectively. We also define
parity as follows.

\begin{proposition}
The superspace structure on the cohomology 
of $\CC[(K_N^\dual\oplus K_N)_0]\otimes \Lambda^*(N_\CC)$
given by the sum of the cohomological and conformal grading
modulo two is coming from the parity of the degree in 
$\Lambda^*(N_\CC)$.
\end{proposition} 

\begin{proof}
Clear.
\end{proof}

\begin{proposition}
The double grading and parity on the $B$ ring given by 
conformal and cohomological gradings coincides with 
the grading on the $B$ ring given in \cite{Krawitz}.
\end{proposition}

\begin{proof}
The double grading on the $B$ ring was first given by Kaufmann, see
\cite{Kau1}-\cite{Kau3}. We will be using its description in \cite{Krawitz}.
Namely, the bi-grading on the $g$-component of the 
$B$ ring is 
$$
(Q^B_+,Q^B_-)=\Big(\sum_{h_j\neq 0} (h_j-q_j),\sum_{h_j\neq 0} (1-h_j-q_j)
\Big)+\Big(p,p\Big) $$
where $p$ is the degree of the polynomial in the Milnor ring.
Here $n_g=\sum_{j=1}^d h_jv_j$ as in the proof of Proposition \ref{B-iso}.

\medskip
The element in \eqref{comesfrom} corresponds to the element 
of the Milnor ring of degree equal to $m\cdot \deg^\dual - \sum_{h_j=0} q_j$,
because we need to adjust for the product of the variables.
As a result, its bi-grading according to \cite{Krawitz} is 
given by 
$$
\Big(\sum_{h_j\neq 0} (h_j-q_j) +
(m\cdot \deg^\dual - \sum_{h_j=0} q_j),
\sum_{h_j\neq 0} (1-h_j-q_j)
+(m\cdot \deg^\dual - \sum_{h_j=0} q_j)
\Big)
$$
$$
=
\Big(\sum_{h_j\neq 0} h_j +
m\cdot \deg^\dual - \sum_{j=1}^d q_j,
-\sum_{h_j\neq 0}h_j + \sum_{h_j\neq 0}1
+ m\cdot \deg^\dual - \sum_{j=1}^d q_j
\Big)
$$$$
=
\Big(\deg\cdot n_g +
m\cdot \deg^\dual - \deg\cdot\deg^\dual,
-\deg\cdot n_g + \sum_{h_j\neq 0}1
+ m\cdot \deg^\dual - \deg\cdot\deg^\dual
\Big)
$$
which equals the cohomological and conformal grading
respectively.
\end{proof}

We can make similar statements about the double grading of the $A$ ring,
based on the comparison of the $A$ and $B$ graded spaces.

\begin{proposition}\label{A-iso}
The $A$ ring of \BH construction, with components twisted
by certain one-dimensional spaces, is naturally isomorphic 
to the cohomology of 
$$
\CC[(K_N^\dual\oplus K_N)_0]\otimes \Lambda^*(M_\CC)
$$
with respect to 
$$
d^A:= \sum_{m\in\Delta} [m]\otimes (\wedge m)
+\sum_{n\in \Delta^\dual} [n]\otimes ({\rm contr.} n).
$$
The bi-grading by $Q^A_+$ and $Q^A_-$ of \cite{Krawitz} comes
from the cohomological and conformal grading on the 
complex respectively. These gradings  assign 
to $[m\oplus n]\otimes P$ the degrees
$m\cdot \deg^\dual + \deg\cdot n-\deg\cdot\deg^\dual$
and $m\cdot \deg^\dual - \deg\cdot n + \deg (P)  -\deg\cdot 
\deg^\dual$ 
respectively.
\end{proposition}

\begin{proof}
It has been observed in \cite{Krawitz} that $A$ and $B$ rings
are isomorphic as vector spaces and their bi-gradings satisfy
$Q^A_+=Q^B_+$
and 
$Q^A_-=d-2\deg\cdot\deg^\dual-Q^B_-$.
The analogous statement for our construction is clear in view of the natural
isomorphism of $\Lambda^*(M_\CC)$ and $\Lambda^*(N_\CC)$.
\end{proof}

\begin{remark}
The aforementioned one-dimensional spaces are rather important.
For example, they may easily switch the parity. From this point
of view, the description of $A$ and $B$  rings of \BH 
construction in terms of complexes, as opposed to Milnor rings,
is more natural. Of course, there is a deeper meaning to these
complexes which will become apparent in the vertex algebra setting.
\end{remark}

\section{Vertex algebra background}
\label{sec.vertprelim}
In this section we will give an informal overview of vertex algebras, $N=2$ structures
on vertex algebras, and the vertex algebras of $\sigma$-model type.
In the latter setting we will define chiral rings. We will also describe 
the lattice vertex algebras
constructed from a pair of dual lattices. All of the material
 can be found elsewhere, in particular in
\cite{chiralrings} but is included here for the benefit of the reader.
It provides the background necessary to understand 
Sections \ref{sec.BB-BH-voa} and \ref{sec.main} which form the heart 
of the paper.

\subsection{Vertex algebras}
A vertex algebra $V$ is a super vector space with an even element $\vert 0\rangle$ 
called \emph{vacuum} 
vector and a rather unusual structure $Y$ called state-field correspondence.
This correspondence $Y:V\to {\rm End}(V)[[z,z^{-1}]]$ 
assigns to every element $a\in V$ a  power series 
$$Y(a,z)=a(z)=
\sum_{l\in\ZZ}a_{(l)}z^{-l-1}$$ 
in the variable $z$ where $a_{(l)}$ are endomorphisms of $V$.
Power series $a(z)$ are called fields and $a_{(l)}$ are called their modes.

\medskip
The correspondence $Y$ must  satisfy a number of axioms, see 
\cite{Kac}.
As a consequence of these axioms,  
every two fields $a(z)$ and $b(z)$ satisfy the \emph{operator product expansion}
(called OPE for short) 
$$
a(z)b(w) = \sum_{j\leq  r} \frac {c_j(w)}{(z-w)^j}
$$
where $r$ is some positive integer and  $c_j(w)$ are other fields of the algebra. 
While all terms of the OPE are useful, one is especially interested in the poles
along $z=w$. Consequently, the above OPE is often abbreviated as 
$a(z)b(w) \sim \sum_{j=1}^r  \frac {c_j(w)}{(z-w)^j}$. 

\medskip
The so-called contour trick  allows one to
read off the supercommutators of the modes of $a(z)$ and $b(z)$ 
from the singular (i.e. $j>0$) part of the 
above OPE. Specifically, the series of supercommutators $[a_{(l)},b(w)]$
is given in the above notations by
$$
[a_{(l)},b(w)]=
{\rm Res}_{z=w}\Big (\sum_{j>0} \frac {z^lc_j(w)}{(z-w)^j}\Big).
$$
In particular, if the above OPE is nonsingular, i.e. $c_j=0$ for $j>0$,
then all modes of $a$ and $b$ supercommute.

\begin{remark}
Vertex algebras are sometimes called chiral algebras, since only holomorphic 
(a.k.a. chiral) variable $z$ is used, as opposed to $z$ and $\bar z$.
\end{remark}

\subsection{Vertex algebras of $\sigma$-model type and chiral rings}
Vertex algebras often come equipped with additional structures. The most common
additional structure is that of a representation of the Virasoro algebra. This means
a choice of an 
even field $L(z)$ of $V$ with a certain OPE with itself that translates into the following
commutator relations for the modes $L[k]:=L_{(k+1)}={\rm Res}_{z=0}L(z)z^{k+1}$
$$
[L[k],L[l]]=(k-l)L[k+l] + \frac c{12}\delta_{k+l}^0(k^3-k)
$$
where $\delta$ is the Kronecker symbol.
These are the relations of the Virasoro algebra with central charge $c$. 
In addition, $L[0]$ is assumed to provide a grading on $V$ and $L[1]$ is assumed
to correspond to differentiation of the fields, see \cite{Kac}.
The $N=2$ structure extends the Virasoro algebra to a somewhat larger superalgebra. 
It is characterized by a choice of three more fields $J(z)$, $G^+(z)$ and $G^-(z)$,
in addition to $L(z)$, which satisfy certain OPEs, see for example \cite{chiralrings}.
It appears naturally in the study of $\sigma$-models with Calabi-Yau target manifolds.
Traditionally, the central charge $c$ is wriiten as $3\hat c$ since $\hat c$ corresponds
to the dimension of the Calabi-Yau. The algebras with such $N=2$ structures
are called $N=2$ vertex algebras of central charge $\hat c$.

\medskip
The following definition lies at the heart of mirror symmetry as it was originally
understood by physicists.
\begin{definition}
For any vertex algebra with an $N=2$ structure, another such structure can be constructed
by switching $G^+$ with $G^-$, sending $J$ to $-J$ and keeping $L$ unchanged.
This involution on the set of $N=2$ structures is called the \emph{mirror involution}.
\end{definition}

The operators $L[0]=L_{(1)}$ and $J_{(0)}$ are of special interest. They commute with
each other and typically provide a double grading on $V$ (although this grading is
not to be confused with the double grading on the $A$ and $B$ rings that was considered
in the  previous section). The following definition was introduced in \cite{chiralrings}
to codify other desirable properties of the $N=2$ vertex algebras.
\begin{definition}
We call an $N=2$ vertex algebra $V$ an $N=2$ vertex algebra of $\sigma$-model type if
the eigenspaces $V_{\alpha}$ of $L_{(1)}$ are finite-dimensional for all $\alpha$ and are
zero except for $\alpha\in \frac 12\ZZ_{\geq 0}$, and the operators 
$H_A:=L_{(-1)}-\frac 12J_{(0)}$
and $H_B:=L_{(1)}+\frac 12 J_{(0)}$ have only nonnegative integer eigenvalues.
\end{definition}

\begin{definition}
The zero eigenspaces of $H_A$ and $H_B$ of $N=2$ algebra $V$ of 
$\sigma$-model type are called the $A$ and $B$ chiral rings of $V$ respectively.
\end{definition}

The following proposition can be found in \cite{LVW}. While the paper overall has
a number of string theory arguments, the proof of this proposition is completely mathematical.

\begin{proposition}\label{propprod}
Let $v_1$, $v_2$ be elements of the $A$ chiral ring of $V$ (or both are in the 
$B$ chiral ring of $V$).
Then their OPE is nonsingular
$$
v_1(z)v_2(w)=v_3(w) + \sum_{j<0} \frac {c_j(w)}{(z-w)^j}.
$$
Moreover, $v_1 \bullet v_2 :=v_3$ defines a supercommutative product on the $A$ chiral 
ring (respectively $B$ chiral ring).
\end{proposition}

\begin{proof}
The idea of the proof is that if the OPE of $v_1$ and $v_2$ 
had a singular part, then this singular part 
would have elements with $H_A<0$ (resp. $H_B<0$). Since this is not possible, the
OPE is nonsingular. Then the product properties follow from general properties of OPEs.
\end{proof}

\begin{remark}
Chiral rings of $N=2$ vertex algebras of $\sigma$-model type are naturally equipped
with \emph{conformal} grading that comes from $2L_{(1)}=\pm J_{(0)}$. 
Vertex algebras $V$ that occur in this paper have an additional finite \emph{cohomological}
grading. This grading should be thought of as some rudimental manifestation of the 
antiholomorphic variables. As a result, the $A$ and $B$ rings of our theory come with a double
grading. 
\end{remark}

\begin{remark}
The paper \cite{LVW} claims the existence of the so-called spectral flow, which in particular
(at the parameter value $1$) gives an isomorphism between $A$ and $B$ chiral
rings of $V$ \emph{as vector spaces}. It comes roughly speaking from 
$\ee^{\int J(z)}$, if such field can be constructed. We will not attempt to axiomatize the
spectral flow, but will rather construct it ad hoc for the vertex algebras under consideration.
\end{remark}

\subsection{Lattice vertex algebras}
We will collect here the standard facts about the lattice vertex algebras which 
form the computational basis of our construction. 
Specifically, for a pair of dual
lattices $M$ and $N$ we want to describe the vertex algebra $\F_{M\oplus N}$.
First, it contains a vertex subalgebra $\F_{{\bf 0}\oplus{\bf 0}}$ which is
generated by even (bosonic) fields $m^{bos}(z),~n^{bos}(z)$ with 
OPEs 
$$
m^{bos}(z)m^{bos}(w)\sim
n^{bos}(z)n^{bos}(w)\sim 0,~~~m^{bos}(z)n^{bos}(w)\sim \frac {m\cdot n}{(z-w)^2}
$$
and odd (fermionic) fields $m^{ferm}(z),~n^{ferm}(z)$ with OPEs
$$
m^{ferm}(z)m^{ferm}(w)
\sim
n^{ferm}(z)n^{ferm}(w)\sim 0, 
$$
$$m^{ferm}(z)n^{ferm}(w)\sim \frac {m\cdot n}{z-w}.
$$
The algebra $\F_{{\bf 0}\oplus{\bf 0}}$ can be thought of as a polynomial ring
in infinitely many even variables $(m_i)^{bos}_{(<0)}$, $(n_i)^{bos}_{(<0)}$
and infinitely many odd variables $(m_i)^{ferm}_{(<0)}$, $(n_i)^{ferm}_{(<0)}$
where $m_i$ and $n_i$ form bases of $M$ and $N$.

\medskip
Note that the endomorphisms $m^{bos}_{(0)}$ and $n^{bos}_{(0)}$ act
by zero on   $\F_{{\bf 0}\oplus{\bf 0}}$. As we pass  to $\F_{M\oplus N}$, we 
introduce other eigenspaces of these operators.
The algebra $\F_{M\oplus N}$ is isomorphic as a vector space 
to $\F_{{\bf 0}\oplus{\bf 0}}\otimes_\CC \CC[M\oplus N]$. There are additional 
operators $\ee^{\int m^{bos}(z)+n^{bos}(z)}$ whose construction is rather 
subtle (but is standard in the field of vertex algebras). 
One needs to use the normal ordering, as well as some explicit cocycle.
The reader is referred to \cite{chiralrings} for more details.

\medskip
The OPEs of $\ee^{\int m^{bos}(z)+n^{bos}(z)}$  with fermionic generators 
are nonsingular. The OPEs of $\ee^{\int m^{bos}(z)+n^{bos}(z)}$ 
with $\tilde m^{bos}(z)$ and $\tilde n^{bos}(z)$ are
$$
\tilde m^{bos}(z)\ee^{\int m^{bos}(w)+n^{bos}(w)} = \frac {\tilde m\cdot n}
{(z-w)}\ee^{\int m^{bos}(w)+n^{bos}(w)}
$$
$$
\tilde n^{bos}(z)\ee^{\int m^{bos}(w)+n^{bos}(w)} = \frac {m\cdot \tilde n}
{(z-w)}\ee^{\int m^{bos}(w)+n^{bos}(w)}.
$$
The OPEs of 
$\ee^{\int m^{bos}(z)+n^{bos}(z)}$  and $\ee^{\int \tilde m^{bos}(z)+\tilde n^{bos}(z)}$  
are obtained from expanding
$$
\ee^{\int m^{bos}(z)+n^{bos}(z)}\ee^{\int \tilde m^{bos}(z)+\tilde n^{bos}(w)}
= \ee^{\int m^{bos}(z) +  \tilde m^{bos}(w)+n^{bos}(z)+\tilde n^{bos}(w)}(z-w)
^{m\cdot \tilde n + \tilde m\cdot n}
$$
in powers of $z-w$.  In particular, the OPE is nonsingular iff  the pairing
$m\cdot \tilde n + \tilde m\cdot n$ is nonnegative. This fact will be used extensively
throughout the calculations of the paper.

\begin{remark}\label{pauli}
A version of Pauli exclusion principle implies that the OPE of $n^{ferm}$ with itself looks
like 
$$
n^{ferm}(z)n^{ferm}(w) = n^{ferm}(w)\partial_w n^{ferm}(w) (z-w)+ O(z-w)^2
$$
because $n^{ferm}(w)n^{ferm}(w)=0$.
\end{remark}

\section{Vertex algebras of Berglund-H\"ubsch mirror symmetry: definitions and
first properties}\label{sec.BB-BH-voa}

In this section, we define the vertex algebras of \BH  mirror symmetry and prove
their first properties. In particular, we extend the Key Lemma of \cite{borvert}
to the \BH setting, see Theorem \ref{keylemma}.

\subsection{Definition of vertex algebras of  Berglund-H\"ubsch mirror symmetry}
Let $W$ and  $G$ be the \BH potential and 
group, and let  $M$, $N$, $K_M$, $K_N$,
$\deg$, $\deg^\dual$, $\Delta$, $\Delta^\dual$ be defined
as in Subsection  \ref{ssec.combref}. Denote by $W^\dual$
and $G^\dual$ the dual potential and the dual group.

Consider the  lattice vertex algebra $\F_{M\oplus N}$
with the $N=2$ structure 
$$
\begin{array}{rcl}
G^+(z)&=&\sum_{i}(n^i)^{bos}(z)(m^i)^{ferm}(z) - \partial_z \deg
^{ferm}(z)\\
G^-(z)&=&\sum_{i}(m^i)^{bos}(z)(n^i)^{ferm}(z) - \partial_z
(\deg^\dual) ^{ferm}(z)\\
J(z)&=&\sum_{i}(m^i)^{ferm}(z)(n^i)^{ferm}(z) +\deg^{bos}(z)
-(\deg^\dual)^{bos}(z)\\
L(z)&=&\sum_{i}(m^i)^{bos}(z)(n^i)^{bos}(z)
+\frac 12\sum_i \partial_z(m^i)^{ferm}(z)(n^i)^{ferm}(z)\\
&&
- \frac 12\sum_i \partial_z(n^i)^{ferm}(z)(m^i)^{ferm}(z)\\
&&
 - \frac 12 \partial _z\deg^{bos}(z) 
 -\frac 12\partial_z(\deg^\dual)^{bos}(z)
\end{array}
$$
with the normal ordering implicitly used in the definition of $J$ and $L$.
This $N=2$ structure has central charge $\hat c=d-2\deg\cdot \deg^\dual$.
Consider the differential 
$$
D_{\bf 1,\bf 1}:={\rm Res}_{z=0}(\sum_{m\in\Delta} m^{ferm}(z)
\ee^{\int m^{bos}(z)}
+\sum_{n\in \Delta^\dual} n^{ferm}(z)
\ee^{\int n^{bos}(z)})
$$
or more generally the differential
$$
D_{f,g}:={\rm Res}_{z=0}(\sum_{m\in\Delta}f(m) m^{ferm}(z)
\ee^{\int m^{bos}(z)}
+\sum_{n\in \Delta^\dual} g(n)n^{ferm}(z)
\ee^{\int n^{bos}(z)})
$$
where $f$ and $g$ are generic coefficient functions.

\begin{definition}\label{vfg}
The vertex algebra $V_{f,g}$ is defined as the cohomology
of $\F_{M\oplus N}$ with respect to the differential $D_{f,g}$.
\end{definition}

\begin{remark}
To check that $D_{f,g}$ is indeed a differential, observe
that its summands  supercommute. This calculation
is based on the property $m\cdot n\geq 0$ for all $m\in\Delta$, $n\in\Delta^\dual$.
For exampe, if $m\cdot n=0$ then the OPE of $m^{ferm}(z)\ee^{\int m^{bos}(z)}$
and $n^{ferm}(w)\ee^{\int n^{bos}(w)}$ has no poles at $z=w$ coming from the bosons
or from fermions. If $m\cdot n>0$ then this OPE has a pole of order one coming from
the fermions, but a zero of order at least one coming from the bosons, so overall
the OPE is nonsingular.
\end{remark}

\begin{proposition}
The vertex algebra $V_{f,g}$ inherits the $N=2$ structure
above from $\F_{M\oplus N}$. 
\end{proposition}

\begin{proof}
One explicitly calculates that $D_{f,g}$ supercommutes with
$G^\pm(z)$. It is important for this calculation that 
$m\cdot \deg^\dual=\deg\cdot n=1$ for all $m\in \Delta$
and $n\in\Delta^\dual$. Since $J$ and $L$ can be calculated from the supercommutators
of $G^\pm$, the statement follows.
\end{proof}

\begin{remark}
Note that when we interchange $(W,G)$ with $(W^\dual,G^\dual)$,
we simply interchange $M$ and $N$ and data therein. 
The corresponding algebra $V$ is then the same, but the $N=2$
structure differs by mirror involution that interchanges $G^+$
and $G^-$, preserves $L$  and sends $J$ to $-J$. There is a 
minor subtlety related to the cocycle in the definition of the vertex
operators, see \cite{borvert} for details, which are still applicable
in the \BH setting.
\end{remark}

\begin{remark}
One can view $\F_{M\oplus N},D_{f,g}$ as a complex, with the grading provided
by $\F_{m\oplus n}\mapsto m\cdot\deg^\dual+\deg\cdot n$. We refer to the induced grading
on $V_{f,g}$
as cohomological. We will later see that it is related to the cohomological grading
on the $A$ and $B$ rings of \BH construction.
\end{remark}

\begin{remark}\label{grading}
The explicit formulas for the grading by $L_{(1)}$ and $J_{(0)}$ are given
in \cite{chiralrings}. They can be briefly described  as follows. 
The $(L_{(1)},J_{(0)})$ grading of $m^{ferm}$ and $n^{ferm}$
is $(\frac 12,1)$ and $(\frac 12,-1)$, 
the grading of $m^{bos}$ and $n^{bos}$ is $(1,0)$.
 Differentiation
changes the grading by $(1,0)$. 
The vertex operator $\ee^{\int m^{bos}(z)+n^{bos}(z)}$
has grading 
$$(m\cdot n + \frac 12 m\cdot \deg^\dual + \frac 12 \deg\cdot n,
-m\cdot \deg^\dual+\deg\cdot n).
$$
\end{remark}

\subsection{Key Lemma}
In this subsection we will prove the \BH analog of the Key Lemma of \cite{borvert}.
We shall first formulate a commutative algebra result
that will be used in the vertex algebra argument. 
Recall that we have dual lattices $M$ and $N$, the
cones $K_M$ and $K_N$ with $K_M\cdot K_N\geq 0$
and the sets $\Delta$ and $\Delta^\dual$ in $M$ and $N$, which 
encode the monomials of $W$ and $W^\dual$ respectively.
We will assume throughout that the coefficient functions 
$f$ and $g$ are generic, which in \BH setting is equivalent to
nonvanishing of all $f(m)$ and $g(n)$.
\begin{proposition}\label{combcond}
For every one-dimensional face of $K_N^\dual$ there 
exists a nonzero lattice point $v$ on it and elements 
$P_n\in \CC[M]$ such that
$$
[v] = \sum_{n\in \Delta^\dual} \Big(P_n
\sum_{m\in\Delta} f(m) (m\cdot n) [m]
\Big)
$$
where $P_n$ have the additional property that all of their monomials
$[w]$ satisfy $w\cdot n\geq -1$ and 
$w\cdot \hat n\geq 0$ for all $\hat n\in\Delta^\dual$ for $\hat n\neq n$.
\end{proposition}

\begin{proof}
Recall that $\CC[K_N^\dual]$ is the $G$-invariant subring of 
the polynomial ring $\CC[x_1,\ldots, x_d]$, see Section \ref{sec.BH-BB-comb}.  
The monomials in the potential $W$ correspond to the
elements  $m\in \Delta$. The rays of $K_N^\dual$ 
correspond to variables $x_i$. Without loss of generality, we may 
assume that the ray in question corresponds to $x_1$. Then 
the monomials $[v]$ correspond to monomials $x_1^l$.
Here $l$ is a multiple of 
some number, determined by the condition that this monomial is 
$G$-invariant.

\medskip
Since the potential $W$ is nondegenerate, the partial 
derivatives $\partial_i W$ form a regular sequence in 
$\CC[x_1,\ldots, x_d]$. As a result, the quotient by these elements
is finite-dimensional. It is also graded (we assign $x_i$ some positive weights $q_i$). 
Thus for all sufficiently large $l$ the monomial
$x_1^l$ lies in the ideal generated by $\partial_i W, i=1,\ldots, d$.
Pick one such $l$ with the additional property that $x_1^l$ is $G$-invariant.
There now exist polynomials $\tilde P_i(x)$
such that 
$$
x_1^l = \sum_{i=1} \tilde P_i(x) \partial_i W
$$
which we can rewrite as 
\begin{equation}\label{eqpart}
x_1^l = \sum_{i=1}P_i(x)  (x_i\partial_i W)
\end{equation}
with $P_i(x)= \frac{\tilde P_i(x)}{x_i}$.
Observe now that $x_i\partial_i W$ is a $G$-invariant polynomial
given by $\sum_{m\in\Delta} f(m) (m\cdot n_i) [m]$ for the corresponding
element $n_i\in\Delta^\dual$.
We can drop all of the monomials in $P_i(x)$ that are not invariant
under the group $G$. We keep calling the resulting polynomials $P_i(x)$
and observe that they correspond to $P_i\in\CC[M]$ that satisfy
 the conditions of the proposition.
\end{proof}

\begin{corollary}\label{corokey}
For every facet of $\theta\subset K_N$ 
there exist $\hat P_n\in\CC[M]$ such that 
$$
[0] = \sum_{n\in \Delta^\dual} \Big(\hat P_n
\sum_{m\in\Delta} f(m) (m\cdot n) [m]
\Big)
$$
where $\hat P_n$ have the following additional property. 
For any $n\in\Delta$ and any $\hat n\in\theta$
all monomials
$[w]$ of $P_n$ satisfy 
$w\cdot \hat n\geq -1$ if $\hat n=n$
and $w\cdot \hat n\geq 0$ if $\hat n\neq n$.
\end{corollary}

\begin{proof}
We simply divide both sides of the result of Proposition
\ref{combcond} by $[v]$ that lies on the ray of $K_N^\dual$ 
which is dual to $\theta$.
\end{proof}

The Key Lemma below  reduces the cohomology
of $\F_{M\oplus N}$ by $D_{f,g}$ to
the cohomology of $\F_{K\oplus N}$ or $\F_{M\oplus K_N}$.
The idea is the following. The operator $D_{f,g}$ is the
sum of two parts, one coming from $\Delta$ and the other
coming from $\Delta^\dual$. These form two differentials
of the double complex, if $\F_{M\oplus N}$ is graded
by assigning $(m\cdot \deg^\dual, \deg \cdot n)$ to
$\F_{(m,n)}\subseteq \F_{M\oplus N}$. However, the
stupid filtrations do not converge. In fact, the cohomology
of the total complex is nontrivial, while the cohomology
with respect to the horizontal or the vertical part of the differential
are both trivial. We will find various homotopies to identity
for one of the differentials, such that their anticommutator
with the other will have certain properties.

\begin{theorem}\label{keylemma}
Let $f$ be a generic coefficient function. Then the
cohomology of $\F_{M\oplus N}$ with respect to $D_{f,g}$ 
is equal to the cohomology of $\F_{M\oplus K_N}$ 
with respect to $D_{f,g}$.
\end{theorem}

\begin{proof}
We follow the method of Propositions 8.1 and  8.2 of \cite{borvert}.
Let $m_0$ be a generator of a ray of $K_N^\dual$ and
$\theta$ the corresponding facet of $K_N$
Consider the operator
on $\F_{M\oplus N}$ given by 
$$
R_\theta :={\rm Res}_{z=0}
 \sum_{n\in \Delta^\dual} n^{ferm}(z)\hat P_n(z)
$$
where $\hat P_n$ is given by Corollary \ref{corokey} in
the sense that every every monomial $[p]$ of $\hat P_n$ is 
converted into $\ee^{\int p^{bos}(z)}$.

\medskip
It is a standard calculation to show that  for any $v\in \F_{M\oplus N}$
$$
R_\theta D_{f,g}v + D_{f,g}R_\theta v  = v+ \alpha_\theta v
$$
where $\alpha_\theta$ increases $m_0\cdot \bullet$ and
does not decrease $\hat m_0\cdot \bullet $ for other generators of 
rays of $K_N^\dual$. Indeed, the definition of $\hat P_n$ 
implies the OPE of 
$$
\sum_{n\in \Delta^\dual} n^{ferm}(z)\hat P_n(z) 
(\sum_{m\in\Delta} f(m) m^{ferm}(w)
\ee^{\int m^{bos}(w)}) \sim \frac 1{z-w}
$$
which in turn implies that the anticommutator of $R_\theta$ with 
the part of $D_{f,g}$ that comes from $M$ is the identity.
As a result, $\alpha_\theta$ is the anticommutator of $R_\theta$
with the $N$-part of $D_{f,g}$. Since $n\in\Delta^\dual$ pair
up nonnegatively with all $\hat m_0$, it remains to show 
that for $\hat n\in \theta$ the corresponding anticommutator is zero.
It then suffices to show that the OPE 
$$
\sum_{n\in \Delta^\dual} n^{ferm}(z)\hat P_n(z) 
\hat n^{ferm}(w)
\ee^{\int \hat n^{bos}(w)} 
$$
is nonsingular. This is assured by the conditions on monomials
of $P_n$. If $n\neq \hat n$, then the OPEs of all of the ingredients
above are nonsingular. If $n=\hat n$, then the OPE of the bosonic
parts may potentially have a pole of order one, along $z=w$
 but it is counteracted
by zero of order one coming from $n^{ferm}(z)n^{ferm}(w)$,
see Remark \ref{pauli}.

\medskip
The argument of Proposition 8.2 of \cite{borvert} now finishes
the proof. 
\end{proof}

\begin{remark}
The difference between \BH and \BB settings is that 
in \BH case we have to worry about the $n=\hat n$ in Corollary \ref{corokey}. 
It corresponds to the difference between a
potential that is nondegenerate in the sense of Batyrev \cite{batduke},
i.e. in the logarithmic coordinates, and the potential that is nondegenerate
in the sense of Berglund-H\"ubsch, i.e. in the usual coordinates.
\end{remark}

\subsection{Fields of the spectral flow.}
In this subsection we will provide two examples of fields of $V_{f,g}$ that will be 
useful later.  They should be thought of as $\ee^{\int \pm J(z)}$.
Consider the fields of $\F_{M\oplus N}$
given by 
$${\rm SF}_+(z)=\ee^{\int \deg^{bos}(z)-(\deg^\dual)^{bos}(z)}\Lambda^d M^{ferm}(z)$$
and
$${\rm SF}_-(z)=\ee^{\int -\deg^{bos}(z)+(\deg^\dual)^{bos}(z)}\Lambda^d N^{ferm}(z).$$
A specific choice of the generator
of $\Lambda^d N^{ferm}(z)$ is not easy to fix, but we can use an
integral basis of $N$ to do so up to $\pm 1$.

\begin{proposition} \label{sf}
Fields ${\rm SF}_\pm(z)$ descend to fields of $V_{f,g}$.
They have cohomological
degree zero. They have $L_{(1)}$ degree $\frac {\hat c} 2$ and 
$J_{(0)}$ degree $\pm \hat c$. In particular, ${\rm SF}_+$  lies in the $A$ chiral ring
of $V_{f,g}$ and ${\rm SF}_-$ lies in its $B$ chiral ring.
\end{proposition}

\begin{proof}
Note that the OPEs of ${\rm SF}_\pm$ with the fields $m^{ferm}(z)\ee^{\int m^{bos}(z)}, m\in\Delta$
and $n^{ferm}(z)\ee^{\int n^{bos}(z)}, n\in\Delta^\dual$ are nonsingular. Indeed,
the poles of order one at $z=w$ for the bosonic part are counteracted by the zeroes
of order one for the fermionic part and vice versa. 
Thus ${\rm SF}_\pm(z)$ supercommute with $D_{f,g}$ and descend to
fields of $V_{f,g}$. Their modes then descend to endomorphisms of $V_{f,g}$.
The degree statements are consequences of the Remark \ref{grading}.
\end{proof}

\section{Main Theorem}\label{sec.main}
In this section we will use the vertex algebra machinery to provide a
natural isomorphism between the $B$ ring of the \BH pair $(W,G)$ 
and the $A$ ring of dual pair $(W^\dual,G^\dual)$. We are working in
the notations of the previous section.

\subsection{Four subcomplexes.} Recall that the vertex algebra 
$V_{f,g}$ is the cohomology of $\F_{M\oplus N}$ with respect to
$D_{f,g}$. We will describe four subcomplexes of $\F_{M\oplus N}$
which are naturally isomorphic to the complexes of Section \ref{sec.V}
when $(f,g)=({\bf 1},{\bf 1})$.

\medskip
Let us denote by ${\mathcal C}_{K_N^\dual-\deg,K_N,N}$ the subspace 
of $\F_{M\oplus N}$ spanned by the elements that corresponds to the fields
of the form
\begin{equation}\label{c1}
\ee^{\int m^{bos}(z)-\deg^{bos}(z)+n^{bos}(z)}
P(\tilde n^{ferm}(z))
\end{equation}
where $m\in K_M, n\in K_M^\dual$ with $m\cdot n =0$ and $P$ is some 
polynomial in the odd anticommuting fields $\tilde n^{ferm}(z), \tilde n\in N$.
Three more subspaces ${\mathcal C}_{K_M-\deg,K_M^\dual,N}$,
${\mathcal C}_{K_N^\dual,K_N-\deg^\dual,M}$ 
and ${\mathcal C}_{K_M,K_M^\dual- \deg^\dual,M}$ are defined similarly.
\begin{proposition}\label{fourC}
The  four subspaces  
$${\mathcal C}_{K_N^\dual-\deg,K_N,N},~{\mathcal C}_{K_M-\deg,K_M^\dual,N},
~
{\mathcal C}_{K_N^\dual,K_N-\deg^\dual,M}
~ {\mathcal C}_{K_M,K_M^\dual- \deg^\dual,M}$$
of  
$\F_{M\oplus N}$ are preserved under the action
of $D_{f,g}$. For $(f,g)=({\bf 1,\bf 1})$, the $D_{f,g}$ cohomology of these complexes 
is naturally isomorphic to the $B$ ring of $(W,G)$, $A$ ring of $(W^\dual,G^\dual)$,
$A$ ring of $(W,G)$ and $B$ ring of $(W^\dual,G^\dual)$ respectively.
\end{proposition}

\begin{proof}
Let us first focus on ${\mathcal C}_{K_N^\dual-\deg,K_N,N}$. 
The action of $D_{f,g}$ is computed by looking at the OPEs.
The OPE of \eqref{c1}
with $m_1^{ferm}(w)\ee^{\int m_1^{bos}(w)}$
for $m_1\in\Delta$ has a pole of order one at $z=w$ coming from the fermionic terms, 
with the residue given by the contraction of $P$ by $m_1$. The action of the corresponding
term of $D_{f,g}$ can thus be nonzero only if the bosonic terms introduce no positive 
powers of $(z-w)$, i.e. $m_1\cdot n=0$ (recall that $n\in K_N\subseteq K_M^\dual$, so
$m_1\cdot n\geq 0)$. Similarly, the OPE of \eqref{c1} with 
$n_1^{ferm}(w)\ee^{\int n_1^{bos}(w)}$ for $n_1\in\Delta$, has no poles coming from the 
fermionic parts and the pole of order at most one coming from the bosons. The pole
of order one is realized when $(m-\deg)\cdot n_1 = -1$ which means $m\cdot n_1=0$.
It is then easy to see that for $(f,g)=({\bf 1},{\bf 1})$ the complex is naturally isomorphic to that of Proposition \ref{B-iso}.

\medskip
The statement about the $B$ ring of $(W^\dual,G^\dual)$ is obtained by switching $M$ 
and $N$. The statements about the $A$ rings follow similarly, by an OPE calculation
and Proposition \ref{A-iso}.
\end{proof}

\begin{remark}
While we are guaranteed that the cohomology spaces of the above four subcomplexes map
to $V_{f,g}$, it is not at all obvious that these maps are injective. It will turn out to be the
case, moreover the images of the cohomology spaces of the first and second subcomplexes
coincide with the $B$ ring of $V_{f,g}$ and the images of the cohomology spaces
of the  third and fourth subcomplexes
coincide with the $A$ ring of $V_{f,g}$, see the proof of Theorem \ref{main} below.
\end{remark}

\subsection{Main theorem}
We are now ready to formulate and prove the main theorem of this paper.
Let $W$ and  $G$ be the \BH potential and 
the group, and let  $M$, $N$, $K_M$, $K_N$,
$\deg$, $\deg^\dual$, $\Delta$, $\Delta^\dual$ be defined
as in Subsection  \ref{ssec.combref}. Denote by $W^\dual$
and $G^\dual$ the dual potential and the dual group.

\begin{theorem}\label{main}
For generic choices of $f$ and $g$
the $N=2$ vertex algebra $D_{f,g}$ is 
of $\sigma$-model type. 
In particular, this algebra is of 
$\sigma$-model type for $(f,g)=({\bf 1},{\bf 1})$.
The $B$ ring of $V_{{\bf 1},{\bf 1}}$ can be identified with 
the $B$ ring of $(W,G)$ and with the $A$ ring of
$(W^\dual,G^\dual)$. The $A$ ring of $V_{{\bf 1},{\bf 1}}$ can
be identified with the $A$ ring of $(W,G)$ and with the $B$ ring of
$(W^\dual,G^\dual)$.
\end{theorem}

\begin{proof}
We will first show that the $B$ and the $A$ rings of $V_{f,g}$ 
are isomorphic to the cohomology spaces of the first and 
the fourth subcomplexes of Proposition \ref{fourC}, 
respectively.

\medskip
Theorem \ref{keylemma} implies that the vertex algebra $V_{f,g}$ 
is the cohomology of $\F_{K_M\oplus N}$ by 
$D_{f,g}$. We will write $D_{f,g}=D_f +D_g$ 
where $D_f$ denotes the terms that come from $M$ and 
$D_g$ denotes the terms that come from $N$.
We now apply Proposition 6.3 of \cite{chiralrings}.
It asserts that for
strongly nondegenerate coefficient functions $f$ the $D_f$-cohomology
of $\F_{K_M\oplus N}$ has $H_A\geq 0$ and moreover
the $H_A=0$ part (i.e. the $A$ chiral ring of $V_{f,g}$) comes from
$\F_{K_M\oplus K_M^\dual-\deg^\dual}$, in the sense that the induced
map on the $H_A=0$ graded pieces of cohomology is an isomorphism.
While Proposition 6.3 of 
\cite{chiralrings} was stated in the setting of \BB mirror symmetry,
it does not use full combinatorial ingredients of the construction
and so it applies to \BH case.

\medskip
Generic $f$ are 
strongly nondegenerate (the definition of strong nondegeneracy
is given in \cite{chiralrings}). Because of the scaling 
symmetries, all $f$ with $f(m)\neq 0, \forall m\in\Delta$, and 
in particular $f={\bf 1}$, have the same behavior and are thus 
strongly nondegenerate. 

\begin{remark}
It is worth mentioning that the proof of Proposition 6.3 of \cite{chiralrings}
is in fact rather easy in the \BH case, since one does not need to
do any fan subdivisions of $K_M$. However, we chose to quote
\cite{chiralrings} since the full power of the argument will be
necessary when we unify the two constructions in the next section.
\end{remark}

Now, as in Proposition 7.3 of 
\cite{chiralrings}, we observe that the $H_A=0$ part 
of $\F_{K_M\oplus K_M^\dual-\deg^\dual}$ is made of linear combinations
of elements that correspond to fields of the form
$$
\ee^{\int m^{bos}(z)+n^{bos}(z)-(\deg^\dual)^{bos}(z)}
P(\tilde m^{ferm}(z))
$$
where $m\in K_M, n\in K_M^\dual$ with $m\cdot n =0$. Thus it comes
from the fourth complex of Proposition \ref{fourC}. By switching $M$ and $N$, we 
see that the $B$ ring of $V_{f,g}$ is calculated by the first complex of 
Proposition  \ref{fourC}. Also, the results of \cite{chiralrings} ensure
that for generic $(f,g)$ the vertex algebra $V_{f,g}$ is of $\sigma$-model type.
The reader may be somewhat concerned about the proof of the statement 
that $L_{(1)}$ graded components of $V_{f,g}$ are finite-dimensional, 
since the argument of \cite{chiralrings} refers to \cite{BLInv}. However, 
one can simply use Proposition 6.5 of \cite{chiralrings} which states that 
a fixed $H_A$ eigenspace of $V_{f,g}$ comes from $\F_{K_M\oplus K_M^\dual-r\deg^\dual}$
for some $r$. If one first takes cohomology with respect
to $D_g$, then we can reduce $m$ to a finite list of possibilities. 
Then it remains to observe that for a fixed $m$ and $r$ the bigraded component
of $\F_{m\oplus  K_M^\dual-r\deg^\dual}$ is finite-dimensional.
Indeed, in the formula of  Remark \ref{grading} we have $m\cdot n\geq -rm\cdot\deg^\dual$,
and thus $\deg\cdot n$ is bonded from above. Thus, we have a finite list 
of possible $n$ as well, and then the fact that $L_{(1)}$ is fixed implies 
the finiteness.

\medskip
We will now work to give alternative descriptions of $A$ and $B$ rings 
of $V_{f,g}$ by means of the third and second complexes of Proposition
\ref{fourC} respectively. We can no longer directly appeal to the results of \cite{chiralrings}.
Instead, we will use the idea of the spectral flow of \cite{LVW}.

\medskip
An easy calculation of  grading, see Remark \ref{grading},
shows that the cohomology of the second 
complex maps inside the $B$ ring of $V_{f,g}$, and the cohomology of 
the third complex maps inside the $A$ ring of $V_{f,g}$. 
On the other hand, there are natural isomorphisms of complexes
between the second and the fourth complexes, as well as the first 
and the third complexes, which we construct below.

\medskip
Consider a $(L_{(1)},J_{(0)})$ degree $(\frac k2,-k)$ element of the 
second complex. It is a linear combination of elements 
$$
\ee^{\int m^{bos}(z)-\deg^{bos}(z)+n^{bos}(z)}
P(\tilde n^{ferm}),~~m\in {K_M},n\in K_M^\dual
$$
with $m\cdot \deg^\dual -\deg\cdot\deg^\dual
-\deg\cdot n+\deg(P)=k$. If we consider the OPE of the above element with the 
spectral flow operator ${\rm SF}_+$ of Proposition \ref{sf}
$$
{\rm SF}_+(z)=\ee^{\int \deg^{bos}(z)-(\deg^\dual)^{bos}(z)}\Lambda^d M^{ferm}(z)
$$
the leading term will be 
$$
(z-w)^{-k}\ee^{\int m^{bos}(w)+n^{bos}(w)-(\deg^\dual)^{bos}(w)}
P^\dual(\tilde m^{ferm}),~~m\in {K_M},n\in K_M^\dual
$$
where $P^\dual$ is obtained by contracting $\Lambda^d M_\CC$ with $P$.
This gives an element of the fourth complex with $(L_{(1)},J_{(0)})$ 
grading $(\frac {\hat c-k}2,\hat c-k)$. There is a natural map 
in the opposite direction given by reading off the leading term of 
OPE with ${\rm SF}_-$. Clearly, these two maps are inverses of 
each other in $\F_{M\oplus N}$. Since ${\rm  SF}_\pm$ 
are fields of $V_{f,g}$ and since reading off the leading term
of the OPE makes sense in $V_{f,g}$, these two maps give inverses
of each other as maps between the image of the cohomology of 
the second complex and the cohomology of the fourth 
complex (i.e. the $A$ ring of $V_{f,g}$). 

\medskip
We have thus constructed an injective map from the $A$ ring of $V_{f,g}$ 
to the $B$ ring of $V_{f,g}$ which sends an element of $(L_{(1)},J_{(0)})$
degree $(\frac l2 ,l)$ to an element of degree $(\frac {\hat c - l}2, l -\hat c)$ by reading
off the $(z-w)^{-l}$ term of OPE with ${\rm SF}_-$. Similarly, there is 
an injective map in the other direction given by reading off the appropriate
term of OPE with ${\rm SF}_+$. Since the dimensions of the $A$ and $B$ 
rings are finite, it means that they coincide. In addition, we see that the 
above spectral flow produces a vector space isomorphism of
$A$ and $B$ rings.

\medskip
It remains to observe that since $({\bf 1},{\bf 1})$ is a pair of strongly nondegenerate
coordinate functions, the above argument works for it. Then Proposition \ref{fourC}
finishes the proof of Theorem \ref{main}.
\end{proof}

\begin{remark}
Since the cohomological grading of ${\rm SF}_\pm$ is zero, the 
spectral flow isomorphism constructed in the proof of Theorem \ref{main}
preserves the cohomological grading. Also observe that the spectral flow
isomorphism corresponds naturally to passing between
$\Lambda^*(N)$ 
and  $\Lambda^*(M)$.
\end{remark}

\begin{remark}
It is important to point out that the isomorphism between the $B$ ring
of $(W,G)$ and the $A$ ring of $(W^\dual,G^\dual)$ constructed in
Theorem \ref{main} is not easy to write explicitly. It means passing
from fields of the form
$$
\ee^{\int m^{bos}(z)-\deg^{bos}(z)+n^{bos}(z)}
P(\tilde n^{ferm}),~~m\in {K_M},n\in K_M^\dual
$$
to the fields of the form
$$
\ee^{\int m^{bos}(z)-\deg^{bos}(z)+n^{bos}(z)}
P(\tilde n^{ferm}),~~m\in {K_N^\dual},n\in K_N
$$
modulo the image of $D_{f,g}$. Presumably, this is what was accomplished
ad hoc in \cite{Krawitz} by using the classification of invertible potentials.
\end{remark}

\section{Unification of Berglund-H\"ubsch and Batyrev-Borisov constructions}
\label{sec.unification}

In this section we describe the unified setting of duality 
that includes both \BB and \BH as special cases. We comment
on necessary combinatorial conditions.  

\subsection{Combinatorial conditions}
We start with dual lattices $M$ and $N$ with chosen
elements $\deg$ and $\deg^\dual$ in them
and collections of elements $\Delta\subset M$ and $\Delta^\dual\subset N$. 
We first require that for all $m\in\Delta$, $n\in \Delta^\dual$ there holds
$\deg \cdot n=m\cdot \deg^\dual=1$ and $m\cdot n\geq 0$.
We define the cones $K_M$ and $K_N$ as being spanned by $\Delta$
and $\Delta^\dual$ respectively. We abuse the notation to use the 
same symbols for the sets of the lattice points in these cones.
We assume that $K_M$ and $K_N$ are of full dimension. 
Their duals are denoted by $K_N^\dual$ and $K_M^\dual$ and we have 
$$
K_N^\dual\supseteq K_M,~K_M^\dual\supseteq K_N.
$$
We also consider the coefficient functions $f:\Delta\to \CC$ and $g:\Delta^\dual\to \CC$.

\begin{definition}\label{unified}
The data $(M,N,\Delta,\Delta^\dual,\deg,\deg^\dual)$ above are said to
give the discrete data of  toric mirror symmetry if the following
two (dual) conditions are satisfied for generic $(f,g)$.
First, for every one-dimensional face of $K_N^\dual$ there 
exists a nonzero lattice point $v$ on it and elements 
$P_n\in \CC[M]$ such that
$$
[v] = \sum_{n\in \Delta^\dual} \Big(P_n
\sum_{m\in\Delta} f(m) (m\cdot n) [m]
\Big)
$$
where $P_n$ have the additional property that all of their monomials
$[w]$ satisfy $w\cdot n\geq -1$ and 
$w\cdot \hat n\geq 0$ for all $\hat n\in\Delta$ for $\hat n\neq n$.
Second, we assume that the dual statement for $K_M^\dual$ and $g$
holds.
\end{definition}

\begin{remark}
We observe that the above conditions are precisely the ones we have proved
in Proposition \ref{combcond} for the \BH case, together with the dual statement.
These conditions also hold for the \BB case where they amount to Batyrev's
nondegeneracy condition from \cite{batduke}.
\end{remark}

We will give a simple restatement of the condition of Definition \ref{unified}
in terms of the homogeneous coordinate rings of \cite{Cox}. For the cone 
$K_N^\dual$ and the set $\Delta^\dual$ consider the polynomial ring 
$\CC[x_i]$ with the number of variables equal to the number of elements in
$\Delta^\dual$. One can encode an element of $K_N^\dual$ by its 
pairings with elements of $\Delta^\dual$. Conversely, a collection of pairings
with elements of $\Delta^\dual$ gives an element of $K_N^\dual$ provided
all the pairings are nonnegative and satisfy some linear relations and/or congruences.
This easy observation identifies $\CC[K_N^\dual]$ as the invariant subring
of $\CC[x_1,\ldots,x_{\sharp(\Delta^\dual)}]$ under the diagonal action of some
abelian algebraic group $G$ of dimension $\sharp(\Delta^\dual)-\rk M$.
\begin{proposition}\label{prop.unified}
Define the potential $W=\sum_{m\in\Delta} f(m)\prod_{n\in\Delta^\dual} x_n^{m\cdot n}$.
Consider the Jacobian ideal ${\rm Jac}(W)$ generated by the partial derivatives
of $W$. Then the first condition of Definition \ref{unified} is equivalent 
to the condition that 
$$
\CC[K_N^\dual]/(\CC[K_N^\dual]\cap {\rm Jac}(W))
$$
is finite-dimensional.
\end{proposition}

\begin{proof}
Suppose that the above quotient is finite-dimensional. It is clearly $M$-graded,
which means that only a finite number of gradings occur. Thus for a ray of $K_N^\dual$
there is an element $v$ on it such that 
$\prod_{n\in\Delta^\dual}x_n^{v\cdot n}\in {\rm Jac}(w)$. The rest of the argument
follows the proof of Proposition \ref{combcond}. 
In the other direction, the condition of Definition \ref{unified} is equivalent to 
the property that $\CC[K_N^\dual]\cap {\rm Jac}(W)$ contains elements $[v]$,
at least one for each ray. Since it is clearly an ideal of $\CC[K_N^\dual]$, 
this implies that the quotient is finite-dimensional.
\end{proof}

\begin{remark}
The following combinatorial condition is necessary, but is perhaps not 
sufficient. For every facet $\theta$  of $K_N$ there exists an element $m\in \Delta$
and an element $n\in \theta$ such that $m\cdot n\leq 1$ 
and $m\cdot \hat n=0$ for all other $\hat n\in \theta$. Indeed, one picks 
$m$ and $n$ in Definition \ref{unified} that contribute nontrivially to 
the graded piece $v$ and use $\theta$ which is dual to the ray through $v$.
Of course, one also needs to impose the dual condition, which does not appear
to follow directly from the original condition.
\end{remark}

\subsection{Vertex algebras of toric mirror symmetry} In the setting of the previous subsection,
we can define vertex algebras $V_{f,g}$ and their chiral rings, for 
the generic choices of $f$ and $g$.

\begin{definition} Let $(M,N,\Delta,\Delta^\dual,\deg,\deg^\dual)$ be 
the discrete data of toric mirror symmetry. Let $f$ and $g$ be coefficient
functions. We define the vertex algebra $V_{f,g}$ as the cohomology 
of the lattice vertex algebra $\F_{M\oplus N}$ by the differential
$$
D_{f,g}:={\rm Res}_{z=0}(\sum_{m\in\Delta}f(m) m^{ferm}(z)
\ee^{\int m^{bos}(z)}
+\sum_{n\in \Delta^\dual} g(n)n^{ferm}(z)
\ee^{\int n^{bos}(z)}).
$$
\end{definition}

Recall the definition of four subcomplexes of $\F_{M\oplus N}$ given
in Proposition \ref{fourC}.
\begin{theorem}\label{univert}
If $f$ and $g$ are generic, then the vertex algebra $V_{f,g}$ is 
of $\sigma$-model type. Its $A$ ring may be calculated either 
as  the cohomology of the subcomplex 
${\mathcal C}_{K_N^\dual,K_N-\deg^\dual,M}$ or the cohomology of 
${\mathcal C}_{K_M,K_M^\dual- \deg^\dual,M}$. Its $B$ ring may be calculated
either as cohomology of ${\mathcal C}_{K_N^\dual-\deg,K_N,N}$
or the cohomology of ${\mathcal C}_{K_M-\deg,K_M^\dual,N}$.
There is a spectral flow isomorphism (as vector spaces) between the 
$A$ ring and the $B$ ring, given by considering the appropriate terms 
of the OPEs with the fields  ${\rm SF}_\pm(z)$  considered in Proposition \ref{sf}.
\end{theorem}

\begin{proof}
We observe that the proof of  Theorem \ref{main} goes through in this
more general situation. Specifically, the condition of Definition \ref{combcond}
assures that the proof of the Key Lemma still works. Then one needs to
further assume that $f$ and $g$ are strongly nondegenerate in the sense
of \cite{chiralrings}, to show that $V_{f,g}$ is of $\sigma$-model type and
to give one calculation of each of the chiral rings. Another calculation
is obtained via the spectral flow isomorphism, as in the proof of Theorem 
\ref{main}.
\end{proof}

\section{Open problems}\label{sec.open}
In this section we comment on the open questions related to our 
construction.

\medskip\noindent
{\bf Classification.} Similar to the \BB situation, one may hope
to classify all $\Delta$ and $\Delta^\dual$ which satisfy the conditions
of Definition \ref{unified}, in the case of $\deg\cdot\deg^\dual=1$
and $\hat c=3$. As in \cite{KS2}, this amounts to a classification of
certain four-dimensional polytopes, although reflexivity condition is 
now somewhat relaxed. Given that the number of nonequivalent 
dimension four reflexive polytopes is around $5\cdot 10^8$
this can be potentially an arduous task. The first step in this direction
would be a reformulation of Definition \ref{unified} in purely combinatorial
terms which do not involve the coefficient functions.

\medskip\noindent
{\bf Frobenius algebra structure.} Our construction endows
the $A$ and $B$ rings of \BH mirror symmetry with the structure
of Frobenius algebra, by evaluating against the top class (a generator
of bidegree $(\hat c,\hat c)$). 

\medskip
It is important to check that the product in our rings coincides with the 
known constructions whenever applicable. While this should be 
reasonably straightforward for \BH models, there appear to be 
serious difficulties in the general setting of Definition \ref{unified}.
At the moment, we can not even prove that the corresponding Hodge diamond
is symmetric. There are also some technical difficulties in comparing
the usual Hessian (used in \BH construction) with the logarithmic Hessian
(for \BB construction). It is conceivable that Definition \ref{unified} may
need to be somehow augmented.

\medskip
It is an open problem in \BB setting to define the product on the chiral
rings without the use of vertex algebras. While the \BH case might be a bit
easier, it is far from obvious. In particular there is a general procedure 
of orbifoldizing Frobenius algebras due to Kaufmann, which often
significantly narrows down the space of possible products, see \cite{Kau3}.
The first step in this direction would be an understanding of the analogs 
of the subrings of Remark \ref{diag}.

\medskip\noindent
{\bf Elliptic genera of Berglund-H\"ubsch models.}
The elliptic genera of \BH models have been calculated in \cite{BeHe}
where it was observed that they satisfy the expected duality.
This invariant should encode the double-graded superdimension of $V_{f,g}$,
which will naturally behave well under passing to the mirror. 
To prove the comparison, one might be able to use 
the method of \cite{BLInv}, which appears to be applicable to the \BH setting.

\medskip\noindent
{\bf Comments on geometry.}
The results of Section \ref{sec.V} may be useful in the setting of 
\cite{Chiodo-Ruan}, where the Landau-Ginzburg model is compared to the 
orbifold theory.  It is plausible that the spaces of the orbifold theory are formally obtained from
those of LG theory by replacing one of the semigroup rings by a partial semigroup ring,
which corresponds to the fan of the ambient toric variety.

\end{document}